\documentclass[a4paper] {article}
[12pt]

\usepackage{setspace}

\usepackage{listings}
\makeatletter
\renewcommand*\l@section{\@dottedtocline{1}{1.5em}{2.3em}}
\makeatother

\usepackage{amsfonts}
\usepackage{amssymb}
\usepackage[T1]{fontenc}

\usepackage{tikz}
\usetikzlibrary{calc}

\usepackage{CJK}
\usepackage{amsmath}
 \usepackage{algorithmicx}
 \usepackage[ruled]{algorithm}
 \usepackage{algpseudocode}

\usepackage{amsfonts}
\usepackage{amssymb}
\usepackage{amsthm}
\usepackage{amssymb}
\usepackage{enumerate}
\usepackage[calc]{picture}
\usepackage[all,cmtip]{xy}

\usepackage[mathscr]{eucal}
\usepackage{eqlist}

\usepackage{color}
\usepackage{abstract}
\usepackage[T1]{fontenc}

\usepackage[top=2.5 cm, bottom=2.5 cm, left=3.5 cm, right=3.5 cm]{geometry}

\theoremstyle{plain}
\newtheorem{theorem}{Theorem}
\newtheorem{proposition}[theorem]{Proposition}
\newtheorem{lemma}[theorem]{Lemma}

\newtheorem{example}[theorem]{Example}
\newtheorem{corollary}[theorem]{Corollary}

\theoremstyle{definition}
\newtheorem{definition}{Definition}

\usepackage{etoolbox}
\newtheoremstyle{myrem}
 {3pt}
 {3pt}
 {\normalsize}
 { }
 {\itshape}
 {:}
 { }
 {}

 \theoremstyle{myrem}
 \newtheorem{remark}{Remark}
 \appto\remark{\leftskip\parindent}
 \appto\remark{\rightskip\parindent}

\numberwithin{equation}{section}
\numberwithin{theorem}{section}

\begin{document}

\begin{center}
{\Large {\textbf {Maps on random hypergraphs and random simplicial complexes }}}
 \vspace{0.58cm}

Shiquan Ren*,  Chengyuan Wu*,  Jie Wu*

\footnotetext[1]{* first authors.  The authors are the first authors contributed equally to the paper.}

\end{center}

\begin{abstract}
Let $L$ be a simplicial complex. In this paper, we study random sub-hypergraphs and random sub-complexes of $L$. By considering the minimal complex that a sub-hypergraph can be embedded in and the maximal complex that can be embedded in a sub-hypergraph, we define some maps on the space of probability functions on sub-hypergraphs of $L$. We study the compositions of these maps as well as their actions on the space of    probability functions.
\end{abstract}



 \section{Introduction}
Random topological objects, for example, random graphs,  random simplicial complexes,   random hypergraphs, etc.,   have important applications in large-data systems in computer science and engineering.  Among these random topological objects, random graphs are the simplest case.  The systematic study of random graphs was started by P. Erd\"{o}s and A. R\'{e}nyi  \cite{1959er,1960er} and  E.N. Gilbert \cite{1959g} around 1960.

Let $0\leq p\leq 1$.
In 1959, P. Erd\"{o}s and A. R\'{e}nyi   \cite{1959er} and E.N. Gilbert \cite{1959g} constructed the Erd\"{o}s-R\'{e}nyi model $G(n,p)$ of random graphs  by   choosing each  pair of  vertices in $V$ as an edge uniformly and independently at random with probability $p$.    In 1960,  thresholds for the connectivity of $G(n,p)$ were given in \cite{1960er}.  In recent decades, the clique complex of $G(n,p)$ was studied in \cite{cfh,clique}.

Random simplicial complexes are higher-dimensional generalizations of random graphs.  In 2006, N. Linial and R. Meshulam  \cite{y2-3} constructed
the Linial-Meshulam model $Y_2(n,p)$ of random $2$-complexes. They take the  complete graph on $V$ and choose each  $2$-simplex of the complete complex on $V$ uniformly and independently at random with probability $p$. The fundamental group   of $Y_2(n,p)$ was studied in \cite{6}. The homology groups of $Y_2(n,p)$ were studied in \cite{y2-4,y2-5}. The asphericity and the hyperbolicity of $Y_2(n,p)$ were studied in \cite{y2-1,y2-2}.

Let $d$ be a non-negative integer. In 2009,  R. Meshulam and  N. Wallach \cite{yd-5} generalized $Y_2(n,p)$  and constructed a model $Y_d(n,p)$  of random $d$-complexes. They  take the $(d-1)$-skeleton of the complete complex on $V$, then choose each  $d$-simplex of the complete  complex on $V$ uniformly and independently at random  with probability $p$.  The homology groups of  $Y_d(n,p)$ were studied in \cite{8,9,7}.  The cohomology of $Y_d(n,p)$ was studied in \cite{yd-1}. Some thresholds for the homology of $Y_d(n,p)$ were given in \cite{annals}. The eigenvalues of the Laplacian on $Y_d(n,p)$ were studied in  \cite{ yd-2}. The collapsibility property of $Y_d(n,p)$ was studied in  \cite{ yd-4,8}. And some sub-structures of $Y_d(n,p)$ were studied in \cite{yd-3}.

Let $0\leq r\leq n-1$ be an integer. Let $ 0\leq p_0,p_1,\ldots,p_{n-1}\leq 1$. Let  $\mathbf{p}=(p_0,p_1,\ldots,p_{n-1})$.  In 2016, $G(n,p)$, $Y_2(n,p)$, $Y_d(n,p)$ and  (the $r$-skeleton of) the clique complex of $G(n,p)$ were generalized universally to a multi-parameter model  of random complexes with probability function ${\rm P}_{n,r,\mathbf{p}}$  by A. Costa and M. Farber \cite{m1,m4}.
In 2017, the fundamental group of the final-generated complexes has been studied in \cite{m2}. The dimension has been studied in \cite{m3}.

Hypergraphs can be obtained by deleting certain faces of  simplcial complexes.
In this paper,   we construct the model of random hypergraphs as an generalization of random simplicial complexes.  We investigate certain maps on random hypergraphs and prove Theorem~\ref{th10}. With the helps of the  maps, we study the relations between random hypergraphs and random simplicial complexes and prove Theorem~\ref{th888}.

The remaining part of this paper is organized as follows.  In Section~\ref{sec999},  we give an outline of the main results.  In Section~\ref{sec2}, we study the map algebra generated by the maps. In Section~\ref{sec3}, we give some geometric characterizations of certain compositions of the maps.   In Section~\ref{sec4}, we prove Theorem~\ref{th10} and Theorem~\ref{th888}.

\subsection{An outline of the main results}\label{sec999}

  Let $n\geq 2$ be an integer. Let $V$ be a   set of $n$ points.  The  power set of $V$, denoted as $2^V$, is the collection of all subsets of $V$. We assume that the emptyset is not in $2^V$ if there is no extra claim.  A  {\it  hypergraph}  $H$ on $V$ is a subset of $2^V$.   In particular, $2^V$ is called the  {\it  complete hypergraph}  and the emptyset $\emptyset$ is called the   {\it  empty hypergraph}.  An element of  $H$ is called  a  {\it hyperedge}, and an element of a hyperedge is called a {\it vertex}. The  {\it  dimension}  of a hyperedge is the cardinality of the hyperedge minus one.  A hyperedge of dimension $d$ is called a {\it  $d$-hyperedge} for short.  The vertex set  of $H$, denoted as $V_H$, is the subset of $V$ consisting of all vertices of all hyperedges in $H$.  A  hypergraph $H'$   on  $V$  is said to be a   {\it  sub-hypergraph}  of $H$  if $H'\subseteq H$.
An   {\it  (abstract) simplicial complex}  $K$ on $V$ is a hypergraph on $V$ such that for any $\sigma\in K$ and any nonempty $\tau\subseteq \sigma$, $\tau\in K$. The hyperedges of $K$ are called {\it simplices}.  A simplicial complex $K'$ is said to be a  {\it  sub-complex}  of $K$ if $K'\subseteq K$.  Given a sub-complex $K'\subseteq K$, a  {\it  $d$-clique}  of   $K'$ in $K$ is a $d$-simplex $\sigma\in K$ such that for any $\tau\subsetneq \sigma$, $\tau\subseteq K'$.

  Let $L$ be a finite simplicial complex.  Let    $\mathcal{H}(L)$ be  the collection of all sub-hypergraphs of $L$.  A  {\it  random sub-hypergraph}  of $L$ is a probability function on $\mathcal{H}(L)$.  Let $D(\mathcal{H}(L))$ be the functional space of all probability functions on $\mathcal{H}(L)$.  Let ${\rm Map}(\mathcal{H}(L))$ be the  semigroup of all self-maps on $\mathcal{H}(L)$.  An element $T\in {\rm Map}(\mathcal{H}(L))$ induces a self-map $DT$ on $D(\mathcal{H}(L))$ by
\begin{eqnarray}\label{e-intro1}
DT(f)(H)=\sum_{TH'=H}f(H'),
\end{eqnarray}
 for any $f\in D(\mathcal{H}(L))$ and any $H\in \mathcal{H}(L)$.
 And a map $F$ from $\mathcal{H}(L)^{\times 2}$ to $\mathcal{H}(L)$ induces a map $DF$ from $D(\mathcal{H}(L))^{\times 2}$ to $D(\mathcal{H}(L))$ by
\begin{eqnarray}\label{e-intro2}
DF(f_1,f_2)(H)=\sum_{F(H_1,H_2)=H}f_1(H_1)f_2(H_2),
\end{eqnarray}
 for any $f_1,f_2\in D(\mathcal{H}(L))$ and any $H\in \mathcal{H}(L)$.
Let $\mathcal{K}(L)$ be the collection of all sub-complexes of $L$. Let ${\rm Map}(\mathcal{K}(L))$ be the semigroup of all self-maps on $\mathcal{K}(L)$.  A {\it  random sub-complex}  of $L$ is a probability function on $\mathcal{K}(L)$.  Let $D(\mathcal{K}(L))$ be the functional space of all probability functions on $\mathcal{K}(L)$.  Similar to (\ref{e-intro1}), an element of ${\rm Map}(\mathcal{K}(L))$ induces a self-map on $D(\mathcal{K}(L))$  and a map from $\mathcal{H}(L)$ to $\mathcal{K}(L)$ induces a map from $D(\mathcal{H}(L))$ to $D(\mathcal{K}(L))$.  And similar to (\ref{e-intro2}), a map $F$ from $\mathcal{K}(L)^{\times 2}$ to $\mathcal{K}(L)$ induces a map $DF$ from $D(\mathcal{K}(L))^{\times 2}$ to $D(\mathcal{K}(L))$.

\begin{definition}\cite{m1,m4}
Let $\Delta_n$ denote the complete simplicial complex on $n$ vertices. Let $\Delta_n^{(r)}$ be the $r$-skeleton of $\Delta_n$. An external face of a sub-complex $Y\subseteq \Delta_n$ is a simplex $\sigma\in\Delta_n$ such that $\sigma\notin Y$ but the boundary of $\sigma$ is contained in $Y$. We use $E(Y)$ to denote the set of all external faces of $Y$. Let $\mathbf{p}=(p_0,p_1,\ldots,p_r)$ with $0\leq p_i\leq 1$. We consider the probability space   $\mathcal{K}(\Delta_n^{(r)})$. The probability function is
\begin{eqnarray*}
{\rm P}_{n,r,\mathbf{p}}(Y)=\prod_{\sigma\in Y,~\dim\sigma\leq r}p_{\dim\sigma}\cdot \prod_{\sigma\in E(Y),~\dim\sigma\leq r}(1-p_{\dim\sigma}).
\end{eqnarray*}
\label{def1}
\end{definition}

The  probability function ${\rm P}_{n,r,\mathbf{p}}$ can be  obtained as follows (cf. \cite[Section~5.5]{contem}):
\begin{enumerate}[(i).]
\item
 We generate the $0$-skeleton     by choosing each vertex of $\Delta_n$ uniformly and independently at random with probability $p_0$.
  \item
  For each $0\leq k\leq r-2$, suppose the $k$-skeleton  is generated. Then we generate the $(k+1)$-skeleton   by choosing each $(k+1)$-clique of the $k$-skeleton in $\Delta_n^{(r)}$ uniformly and independently at random with probability $p_{k+1}$.
  \item
  The final-generated complexes have  probability function ${\rm P}_{n,r,\mathbf{p}}$.  Thus
\begin{eqnarray*}
\sum_{Y\subseteq \Delta_n^{(r)}}{\rm P}_{n,r,\mathbf{p}}(Y)=1.
\end{eqnarray*}
  \end{enumerate}

\smallskip

Let $p: L\to [0,1]$ be an arbitrary function.  In the next definition, we generalize Definition~\ref{def1}  and give a model of  random sub-complex   in  $L$.

 \begin{definition}
[Generalization of Definition~\ref{def1}]
An external face of a sub-complex $Y\subseteq L$ is a simplex $\sigma\in L$ such that $\sigma\notin Y$ but the boundary of $\sigma$ is contained in $Y$. We
use $E(Y)$ to denote the set of all external faces of $Y$ in $L$.  We consider the probability space  $\mathcal{K}(L)$. The probability function is given by
\begin{eqnarray*}
{\rm P}_{L, {p}}(Y)=\prod_{\sigma\in Y}p(\sigma)\cdot \prod_{\sigma\in E(Y)}\big(1-p(\sigma)\big).
\end{eqnarray*}
In particular, suppose $\dim L=r$ and there exists $0\leq p_0,p_1,\ldots,p_r\leq 1$ such that for each $\sigma\in L$,  $p(\sigma)=p_{\dim\sigma}$. Then we
denote ${\rm P}_{L, {p}}$ as ${\rm P}_{L,\mathbf{p}}$. We have ${\rm P}_{\Delta_n^{(r)},\mathbf{p}}= {\rm P}_{n,r,\mathbf{p}}$.
\label{def2}
\end{definition}
 The random complex  model in  Definition~\ref{def2} can be generated as follows:
 \begin{enumerate}[(i).]
 \item
  Choose each vertex $v\in L$   independently at random with probability $p(v)$.
\item
 For each $0\leq k\leq \dim L-1$, suppose the $k$-skeleton  is generated. Then we generate the $(k+1)$-skeleton   by choosing each $(k+1)$-clique $\sigma$ of the $k$-skeleton in $L$   independently at random with probability $p(\sigma)$.
 \item
The final-generated complexes have the probability function ${\rm P}_{L,p}$. Hence
  \begin{eqnarray*}
\sum_{Y\subseteq L}{\rm P}_{L,{p}}(Y)=1.
\end{eqnarray*}
\end{enumerate}

\smallskip

 In the next definition, we consider an analogue of Definition~\ref{def2}  and give a model of  random sub-hypergraph   in  $L$.

 \begin{definition}
[Hypergraphic analogue  of Definition~\ref{def2}]
 We consider the probability space   $\mathcal{H}(L)$. The probability function is given by
\begin{eqnarray*}
\bar{{\rm P}}_{L,{p}}(H)=\prod_{\sigma\in H}p(\sigma)\cdot \prod_{\sigma\notin H}\big(1-p(\sigma)\big).
\end{eqnarray*}
In particular, suppose $\dim L=r$ and there exists $0\leq p_0,p_1,\ldots,p_r\leq 1$ such that for each $\sigma\in L$,  $p(\sigma)=p_{\dim\sigma}$. Then we
denote $\bar{{\rm P}}_{L, {p}}$ as $\bar{{\rm P}}_{L,\mathbf{p}}$.
\label{def3}
\end{definition}

 The  random hypergraph in Definition~\ref{def3} can be  generated as follows.
 We choose each  simplex $\sigma\in L$   independently at random with probability $p(\sigma)$.   We obtain a  hypergraph.  The probability function of these independent trials is $\bar{{\rm P}}_{L,{p}}$. Therefore,
 \begin{eqnarray*}
\sum_{H\subseteq L}\bar{{\rm P}}_{L,{p}}(H)=1.
\end{eqnarray*}

We let $H\in \mathcal{H}(L)$.  We study the minimal  complex  $\Delta H$ that $H$ can be embedded in, the maximal  complex $\delta H$ that can be embedded in $H$, and the complement hypergraph $\gamma H$  in $L$. By composing $\Delta,\delta$ and $\gamma$ iteratively, we obtain a sub-semigroup $G$ of ${\rm Map}(\mathcal{H}(L))$. And $G$ induces a semi-group $DG$ of self-maps on $D(\mathcal{H}(L))$.  Moreover,  by composing $\Delta\gamma$ and $\delta\gamma$  iteratively, we obtain a sub-semigroup $G'$ of ${\rm Map}(\mathcal{K}(L))$.  And    $G'$ induces a semi-group $DG'$ of self-maps on $D(\mathcal{K}(L))$. We   study the  map algebra acting on $D(\mathcal{H}(L))$ induced from $\Delta$, $\delta$ and $\gamma$, and the map algebra acting on $D(\mathcal{K}(L))$ induced from $\Delta\gamma$ and $\delta\gamma$.    In particular, we give some explicit expressions for the actions of the map algebra on $\bar{{\rm P}}_{L,p}$ and ${{\rm P}}_{L,p}$.  As consequences, we give algorithms generating large sparse random hypergraphs with probability function $\bar{{\rm P}}_{\Delta_n,\mathbf{p}}$, and algorithms generating large sparse random simplicial complexes with probability function ${{\rm P}}_{\Delta_n,\mathbf{p}}$.

\smallskip

Let $\sigma\in L$. The characteristic probability $\varphi_\sigma$ is the function
\begin{eqnarray*}
\varphi_\sigma(\sigma')=\left\{
   \begin{aligned}
  0,  {\rm  \ \ \ if ~}\sigma'\neq\sigma;\\
  1, {\rm  \ \ \ if ~}\sigma'=\sigma.
   \end{aligned}
   \right.
\end{eqnarray*}
A {\it path} $s$ in $L$  is a sequence of simplices $\sigma_1 \sigma_2 \ldots \sigma_m$ in $L$ such that the intersection of any two consecutive simplices is nonempty.  We call $m$ the length of $s$. Given two simplices $\sigma,\sigma'\in L$,  the {\it  distance} between $\sigma$ and $\sigma'$ is
\begin{eqnarray*}
d(\sigma,\sigma')=\min\{m \mid s=\sigma_1\sigma_2\ldots\sigma_m {\rm  ~is~ a~ path~ in~ }L, ~\sigma_1=\sigma, ~\sigma_m=\sigma'\}.
\end{eqnarray*}
The {\it  diameter} of  $L$ is
${\rm diam}L=\max _{\sigma,\sigma'\in L}d(\sigma,\sigma')$.
 Let  $m=\max_{\sigma,\sigma'\in\max(L)}d(\sigma,\sigma')$.
 The first main result  of this paper is  the next   Theorem.

\begin{theorem}[Main Result I]\label{th10}
Let   $k$ be a non-negative integer. Let ${\rm Ext}=\Delta\gamma\delta\gamma$ and ${\rm Int}=\delta\gamma\Delta\gamma$. Let $f\in D(\mathcal{H}(L))$.
\begin{enumerate}[(a).]
\item
 If   $k\geq  {\rm diam}L$, then
$(D{\rm Ext})^{k}(f)=f(\emptyset)\varphi_\emptyset+ \big(1-f(\emptyset)\big)\varphi_L$.
\item
 If  $k\geq   {\rm diam}L$,    then
$(D{\rm Int})^{k}(f)=\big(1-f(L)\big)\varphi_\emptyset+ f(L)\varphi_L$.
 \item
There exists  $f\in D(\mathcal{H}(L))$  such that
$$
f, (D{\rm Ext})(f), (D{\rm Ext})^{2}(f), \ldots, (D{\rm Ext})^{m-1}(f), (D{\rm Ext})^{m}(f)
$$
are distinct;
\item
There exists $f\in  D(\mathcal{H}(L))$  such that
$$
f, (D{\rm Int})(f), (D{\rm Int})^{2}(f), \ldots, (D{\rm Int})^{m-1}(f), (D{\rm Int})^{m}(f)
$$
are distinct;
\item
Let $k\geq 1$. Then  for any probability function $f\in D(\mathcal{H}(L))$, the probability that ${\rm Ext}^{k-1}(\gamma H)\subseteq \gamma {\rm Int}^k(H)$  and  the probability that ${\rm Int}^k(H)\subseteq {\rm Ext}^{k+1}(\gamma H)$
 are $1$;

\item
 For any probability function $f\in D(\mathcal{H}(L))$, the probability that $\Delta H\subseteq {\rm Int}\circ{\rm Ext}(H)$ is greater than or equal to  $\sum_{V_H\subseteq H} f(H)$ where the condition $V_H\subseteq H$ means that each vertex of $H$  is a $0$-hyperedge.
\end{enumerate}
\end{theorem}

Consider the spaces of probability functions
\begin{eqnarray*}
&F(\mathcal{H}(L))=\{\bar{{\rm P}}_{L, {p}}\mid p: L\longrightarrow [0,1]\},\\
&F(\mathcal{K}(L))=\{{{\rm P}}_{L, {p}}\mid p: L\longrightarrow [0,1]\}.
\end{eqnarray*}
 Then $F(\mathcal{H}(L))$ is a subspace of $D(\mathcal{H}(L))$ and  $F(\mathcal{K}(L))$ is a subspace of $D(\mathcal{K}(L))$.   Let $\cap$ and $\cup$ be the intersection   and  the union of hypergraphs.
The second main result of this paper is the next theorem.

\begin{theorem}[Main Result II]\label{th888}
The map $D\gamma$ maps $F(\mathcal{H}(L))$ to itself. The maps $D\Delta$ and $D\delta$ map   $F(\mathcal{H}(L))$ to  $F(\mathcal{K}(L))$. The map  $D\cap$ maps  $F(\mathcal{H}(L))^{\times 2}$ to $F(\mathcal{H}(L))$,  and maps  $F(\mathcal{K}(L))^{\times 2}$ to $F(\mathcal{K}(L))$.  And the map  $D\cup$ maps
 $F(\mathcal{H}(L))^{\times 2}$ to $F(\mathcal{H}(L))$. Precisely,

 \begin{enumerate}[(a). ]
 \item
 $D\gamma$ sends $\bar{{\rm P}}_{L,p}$ to $\bar{{\rm P}}_{L,1-p}$;
 \item
 $D\Delta$ sends $\bar{{\rm P}}_{L,p}$ to
  a  random sub-simplicial complex  of $L$  given by
\begin{eqnarray}\label{th888-e1}
\big(D\Delta(\bar{{\rm P}}_{L,p})\big)(K)= \Big(\prod_{\tau\in {\rm max}(K)} p(\tau)\Big)\Big(\prod_{\tau\notin  K}\big(1-p(\tau)\big)\Big)
\end{eqnarray}
for any  $K\in \mathcal{K}(L)$;

 \item
 $D\delta$ sends $\bar{{\rm P}}_{L,p}$ to
  a  random sub-simplicial complex  of $L$  given by
 \begin{eqnarray}\label{th888-e2}
\big  ( D\delta(\bar{{\rm P}}_{L,p})\big)(K)= \sum_{\delta H=K} \prod_{\sigma\in H} p(\sigma) \prod_{\sigma\notin H} \big(1-p(\sigma)\big)
\end{eqnarray}
for any  $K\in \mathcal{K}(L)$;

 \item
$D\cap$ sends     the pair $(\bar{{\rm P}}_{L,{p}'},\bar{{\rm P}}_{L,{p}''})$ to $\bar{{\rm P}}_{L,{p}'p''}$;

\item
$D\cup$ sends the pair $(\bar{{\rm P}}_{L,{p}'},\bar{{\rm P}}_{L,{p}''})$ to $\bar{{\rm P}}_{L,r,1-(1-p')(1-p'')}$.

 \end{enumerate}
\end{theorem}

 \section{Map algebras on hypergraphs and simplicial complexes}\label{sec2}

 In this section, we study the minimal  complex  $\Delta H$ that $H$ can be embedded in, the maximal  complex $\delta H$ that can be embedded in $H$, and the complement hypergraph $\gamma H$  in $L$. We study the map algebras of the compositions of $\Delta,\delta$ and $\gamma$ as well as the intersections and unions. We also study the restrictions of the compositions of $\Delta,\delta$ and $\gamma$ on simplicial complexes.

 \subsection{The map algebra on hypergraphs}\label{s2.1}

We consider the maps
$\Delta,\delta: \mathcal{H}(L)\longrightarrow \mathcal{K}(L)$, and $\gamma : \mathcal{H}(L)\longrightarrow \mathcal{H}(L)$ given by
\begin{eqnarray*}
\Delta H&=&\{\sigma\in L\mid {\rm ~there~ exists~ }\tau\in H{\rm ~ such ~that~ } \sigma\subseteq \tau\};\\
\delta H&=&\{\sigma\in L\mid {\rm ~ for~ any~ }\tau\subseteq \sigma,~ \tau\in H\};\\
\gamma H&=&\{\sigma\in L\mid \sigma\notin H\}
\end{eqnarray*}
for any  $H\in \mathcal{H}(L)$.
Then
(i). $\gamma^2={\rm id}$;
(ii). $\Delta\delta=\delta$;
(iii). $\delta\Delta=\Delta$;
(iv). $\Delta^2=\Delta$;
(v).  $\delta^2=\delta$;
(vi). $(\Delta\gamma\Delta\gamma)^2=\Delta\gamma\Delta\gamma$;
(vii). $(\delta\gamma\delta\gamma)^2=\delta\gamma\delta\gamma$.
The equalities (i) - (v) are  straight-forward.    Let $\max(L)$ be the set of all maximal faces of $L$. We   prove (vi) and (vii).

\begin{proof}[Proof of  (vi)]
 Let $H\in\mathcal{H}(L)$. Then
\begin{eqnarray}
\Delta\gamma\Delta\gamma H&=&\Delta\gamma\Delta\{\sigma\in L\mid \sigma\notin H\}\nonumber\\
&=&\Delta\gamma\{\sigma\in L\mid{\rm~ there~ exists~ }\sigma\subseteq\tau {\rm ~ such ~that~ } \tau\notin H\}\nonumber\\
&=&\Delta\{\sigma\in L\mid {\rm ~there~ does ~not~ exist~  any ~}\sigma\subseteq\tau {\rm  ~such~ that~ }\tau\notin H\}\nonumber\\
&=&\Delta\{\sigma\in L\mid {\rm ~for ~any~ }\sigma\subseteq\tau,~\tau\in H\}\label{e-int-0}\\
&=&\Delta\{\sigma\in \max(L)\mid  \sigma\in H\}\nonumber\\
&=&\Delta\big(\max(L)\cap H\big). \nonumber
\end{eqnarray}
Thus
\begin{eqnarray*}
(\Delta\gamma\Delta\gamma)^2 H
=\Delta \Big( \max(L)\cap \Delta\big(\max(L)\cap H\big)\Big)
=\Delta\big(\max(L)\cap H\big).
\end{eqnarray*}
Since $H$ is arbitrary, we have (vi).
\end{proof}

\begin{proof}[Proof of (vii)]
 Let $H\in\mathcal{H}(L)$. Then
\begin{eqnarray*}
\delta\gamma\delta\gamma H&=&\delta\gamma\delta\{\sigma\in L\mid \sigma\notin H\}\\
&=&\delta\gamma\{\sigma\in L\mid {\rm ~for~ any~ }\tau\subseteq \sigma, ~ \tau\notin H\}\\
&=&\delta\{\sigma\in L\mid{\rm ~there~ exists~ }\tau\subseteq\sigma,~ \tau\in H\}\\
&=&\{\sigma\in L\mid{\rm ~for~ any~ }\sigma'\subseteq\sigma, {\rm ~there~ exists~ }\tau\subseteq\sigma', ~\tau\in H\}\\
&=&\{\sigma\in L\mid {\rm ~for~ any~ vertex~ } v {\rm ~ of ~}\sigma,~ \{v\} \in H\}.
\end{eqnarray*}
Hence $\delta\gamma\delta\gamma H$ is the sub-complex of $L$ spanned by all the $0$-hyperedges in $H$. And $(\delta\gamma\delta\gamma)^2 H$ is the sub-complex of $L$ spanned by all the $0$-hyperedges in $\delta\gamma\delta\gamma H$. Since the $0$-hyperedges of $H$ and the $0$-hyperedges of $\delta\gamma\delta\gamma H$ are same, we have
$(\delta\gamma\delta\gamma)^2 H=\delta\gamma\delta\gamma  H$.
Since $H$ is arbitrary, we have (vii).
\end{proof}

Let $G$ be
the  semi-group generated by $\Delta, \delta,\gamma$ modulo the relations (i) - (vii).  The multiplication of $G$ is the composition of maps.  The unit of $G$ is ${\rm id}$, the identity map on $\mathcal{H}(L)$.  There are four types of elements in $G$:
(1).
$\gamma x_1\gamma x_2\ldots \gamma x_k\gamma$;
(2).
 $x_1\gamma x_2\ldots \gamma x_k\gamma$;
(3).
 $\gamma x_1\gamma x_2\ldots \gamma x_k$;
(4).
 $  x_1\gamma x_2\ldots \gamma x_k $.
\noindent In (1)  - (4), $k$ is a nonnegative integer,  $x_i=\Delta$ or $\delta$, and there does not exist any  $4$ consecutive $x_i$'s that take the same value.
For any $w_1,w_2\in G$ and any  $H_1,H_2\in \mathcal{H}(L)$, let
\begin{eqnarray*}
(w_1+w_2)(H_1,H_2)&=&w_1(H_1)\cup w_2(H_2); \\
(w_1\wedge w_2)(H_1,H_2)&=&w_1(H_1)\cap w_2(H_2).
\end{eqnarray*}
 For any  positive integer $t$, let
 \begin{eqnarray*}
 G^t&=&\{(\ldots (w_1*w_2)*\ldots *w_t)\mid *=\wedge {\rm  or }+,  w_1,w_2,\ldots,w_t\in G\\
  && {\rm  ~with~ any~ } t-2 {\rm~  brackets~ } (\cdot ) {\rm  ~giving ~the ~order~ of~ evaluation}\}.
 \end{eqnarray*}
For any $W\in G^t$, $W$ is a map from $\mathcal{H}(L)^{\times t}$ to $\mathcal{H}(L)$.
Some relations among $w\in G$, $+$ and $\wedge$ are:
\begin{enumerate}[(I).]
\item\label{e}
$(w_1\wedge w_2)\wedge w_3=w_1\wedge (w_2\wedge w_3)$;
\item\label{rel-b}
$(w_1+ w_2)+w_3=w_1+ (w_2+ w_3)$;
\item\label{j}
$\gamma(w_1+w_2)=(\gamma w_1)\wedge (\gamma w_2)$, or equivalently, $\gamma (w_1\wedge w_2)=\gamma w_1 + \gamma w_2$;
\item\label{j-1}
$\Delta(w_1+w_2)=(\Delta w_1) + (\Delta w_2)$;
\item\label{j-3}
$\delta(w_1\wedge w_2)=\delta w_1 \wedge \delta w_2$.
\end{enumerate}
  Here  $w_1,w_2,w_3\in G$,
and $0$ is the constant map sending  $\mathcal{H}(L)$ to $\emptyset$.
(I) - (III) are straight-forward. Let $H,H'\in \mathcal{H}(L)$. Let $H_1=w_1(H)$ and $H_2=w_2(H')$.  We prove (\ref{j-1}) and (\ref{j-3}).

\begin{proof}[Proof of (\ref{j-1})] In order to prove (\ref{j-1}), we only need to prove that for any $H_1, H_2\in \mathcal{H}(L)$,
$\Delta(H_1\cup H_2)=\Delta H_1\cup\Delta H_2$.
Let $\sigma\in L$. Then $\sigma\in \Delta(H_1\cup H_2)$ iff. there exists $\tau\in H_1\cup H_2$ such that $\sigma\subseteq \tau$. This happens iff. there exists $\tau\in H_1$ such that $\sigma\subseteq\tau$ or there exists $\tau\in H_1$ such that $\sigma\subseteq\tau$.   Hence $\sigma\in \Delta(H_1\cup H_2)$ iff. $\sigma\in \Delta H_1$ or $\sigma\in\Delta H_2$, that is, $\sigma\in \Delta H_1\cup\Delta H_2$.
\end{proof}

\begin{proof}[Proof of (\ref{j-3})]
In order to prove (\ref{j-3}), we only need to prove that for any $H_1, H_2\in \mathcal{H}(L)$,
$\delta(H_1\cap H_2)=\delta H_1\cap\delta H_2$.
Let $\sigma\in L$. Then $\sigma\in \delta(H_1\cap H_2)$ iff.  for any $\tau\subseteq\sigma$, $\tau\in H_1\cap H_2$. This happens iff.  for any $\tau\subseteq\sigma$, $\tau\in H_1$ and  $\tau\in H_2$. That is, $\sigma\in \delta H_1\cap\delta H_2 $.
\end{proof}

\begin{remark}
Similar to the proofs of (\ref{j-1}) and (\ref{j-3}),  for any  $H_1, H_2\in \mathcal{H}(L)$,
$\Delta(H_1\cap H_2) \subseteq  \Delta H_1\cap\Delta H_2$ and
$\delta(H_1\cup H_2) \supseteq  \delta H_1\cup \delta H_2$.
Hence for any $H,H'\in\mathcal{H}(L)$,
\begin{eqnarray*}
\big(\Delta(w_1\wedge w_2)\big)(H,H')&\subseteq& (\Delta w_1\wedge \Delta w_2)(H,H'), \\
\big(\delta(w_1+ w_2)\big)(H,H')&\supseteq& (\delta w_1+ \delta w_2)(H,H').
\end{eqnarray*}
\end{remark}

 \subsection{The map algebra on simplicial complexes}

Let $w\in G$. A subset $S$ of $\mathcal{H}(L)$ is called an invariant subspace of $w$ if for any $H\in S$, $w(H)\in S$. We  consider   $w\in G$  such that $\mathcal{K}(L)$ is an invariant subspace of $w$.  The collection of all such $w$ forms a subgroup $G_1$ of  $G$.   Since both $\Delta$ and $\delta$ act on $\mathcal{K}(L)$ identically, we take an equivalent relation $\sim$ identifying both $\Delta$ and $\delta$ as the unit element of $G_1$. We denote the quotient group $G_1/\sim$ as $G'$.

Precisely, $G'$ can be constructed as follows.
Let $\alpha=\Delta\gamma$ and $\beta=\delta\gamma$.  Let $G'$ be
the  semi-group generated by $\alpha$ and $\beta$ modulo the relations
 (i)'. $\alpha^4=\alpha^2$;
 (ii)'. $\beta^4=\beta^2$.
The multiplication of $G'$ is the composition of maps. The unit of $G'$ is ${\rm id}$, the identity map on $\mathcal{K}(L)$.
The elements in $G'$ are:
(1)'. $\alpha^{m_1}\beta^{n_1}\ldots \alpha^{m_k}\beta^{n_k}$;
(2)'. $\beta^{n_1}\alpha^{m_2}\ldots \alpha^{m_k}\beta^{n_k}$;
(3)'. $\alpha^{m_1}\beta^{n_1}\ldots \beta^{n_{k-1}}\alpha^{m_k}$;
(4)'. $ \beta^{n_1}\alpha^{m_2}\ldots \beta^{n_{k-1}}\alpha^{m_k} $.
Here $k$ is a nonnegative integer and  $1\leq m_i,n_i\leq 3$, $i=1,2,\ldots$
Similar to Subsection~\ref{s2.1}, we  define $+$ and $\wedge$ for the elements in $G'$. We construct $G'^t$ for all positive integer $t$.  Each element $W'\in G'^t$ is a map from $\mathcal{K}(L)^{\times t}$ to $\mathcal{K}(L)$.

\section{Some characterizations of the maps}\label{sec3}

Let ${\rm Ext}=\Delta\gamma\delta\gamma$ and ${\rm Int}=\delta\gamma\Delta\gamma$.  In this section, we study the properties of the maps ${\rm Ext}$ and ${\rm Int}$. We give some geometric characterizations of ${\rm Ext}$ and ${\rm Int}$. In Subsection~\ref{s3.1}, we use paths in complexes to characterize the powers of ${\rm Ext}$ and ${\rm Int}$. In Subsection~\ref{s3.2}, we use neighborhoods of sub-complexes to study ${\rm Ext}$, ${\rm Int}$ and their compositions.

\subsection{Powers of maps and paths}\label{s3.1}

 Let $H\in \mathcal{H}(L)$.  Let ${\rm Ext}(H)=\Delta\gamma\delta\gamma(H)$. Then
  \begin{eqnarray*}
 {\rm Ext}(H)&=&\Delta\{\tau\in L\mid {\rm~ there~ exists~ } \sigma\in H {\rm  ~ such~ that~ }\sigma\subseteq\tau\}\\
 &=&\{\tau\in L\mid {\rm ~there~ exists~ } \tau\subseteq\tau' \\
 &&{\rm ~ such~ that~  there~ exists ~} \sigma\in H{\rm ~ with~ }\sigma\subseteq\tau'\}.
 \end{eqnarray*}
  Hence ${\rm Ext}(H)$ is the sub-complex of $L$ obtained by extending each hyperedge $\sigma$ of $H$ to a maximal face of $L$ containing $\sigma$.  We call ${\rm Ext}(H)$ the {\bf extension} of $H$.
We notice that every maximal face of ${\rm Ext}(H)$ is in $\max (L)$, and
\begin{eqnarray}
{\rm Ext}(H)&=&\Delta\big(\max(L)\cap{\rm Ext}(H)\big)\nonumber\\
&=&\Delta\{\tau_1\in\max(L)\mid {\rm  ~there~ exists~ }\sigma\in H{\rm  ~such~ that~ }\sigma\subseteq \tau_1\}. \label{e-ext-0}
\end{eqnarray}
   For any $k\geq 2$, by an induction on $k$ and  (\ref{e-ext-0}),
    \begin{eqnarray}
  {\rm Ext}^k(H)
   &=&\Delta\{\tau_k\in\max(L)\mid {\rm  ~there~ exists~ }\tau_1,\tau_2,\ldots,\tau_{k-1}\in\max(L) {\rm~  and~ }\sigma\in H \nonumber\\
   &&{\rm  ~such ~that~  }\tau_i\cap \tau_{i-1}\neq\emptyset {\rm ~ for~ any~ } 2\leq i\leq k {\rm  ~and~ }\sigma\subseteq \tau_1\}. \label{e-ext-2}
   \end{eqnarray}

A path $s=\sigma_1\sigma_2\ldots\sigma_m$ in $L$ is called a {\bf broad path} if for each $1\leq i\leq m$, $\sigma_i$ is a maximal face of $L$. For any path $s$ in $L$, if we extend each $\sigma$ of $s$ to be a maximal face $\tau\in \max(L)$ such that $\sigma\subseteq\tau$, then we obtain a broad path $s'$.

      \begin{lemma}
  Let $\sigma,\sigma'\in \max(L)$ with $d(\sigma,\sigma')=n$.   Then   there exists a broad path of length $n$ starting from $\sigma$ and ending at $\sigma'$.
    \end{lemma}
    \begin{proof}
  Since $d(\sigma,\sigma')=n$, there exists a path $s=\tau_1\tau_2\ldots\tau_n$ in $L$ such that $\sigma=\tau_1$ and $\sigma'=\tau_n$.   For each $1\leq i\leq n$, we extend $\tau_i$ to be a maximal face $\tau'_i$ of $L$. Then we obtain the broad path $s'=\tau'_1\ldots\tau'_n$.
    \end{proof}

 Let ${\rm Int}(H)=\delta\gamma\Delta\gamma(H)$. Then  with the help of (\ref{e-int-0}),
  \begin{eqnarray}
 {\rm Int}(H)&=&\delta\{\tau\in L\mid {\rm~ for~ any~ }\tau\subseteq\sigma,~ \sigma\in H\}\nonumber\\
 &=&\{\tau\in L\mid {\rm ~for ~any~ }\tau'\subseteq\tau {\rm ~ and ~ any~ }\tau'\subseteq\sigma,~ \sigma\in H\}\nonumber\\
 &=&\{\tau\in L\mid {\rm~ for~ any~ }\sigma {\rm  ~with~ }\sigma\cap\tau=\tau'\neq\emptyset,~ \sigma\in H\}.
 \label{e-int-8}
 \end{eqnarray}
  Hence ${\rm Int}(H)$ is the sub-complex of $L$ consisting of all the hyperedges $\tau \in H$ such that for any $\sigma\in \gamma H$, $\sigma\cap \tau$ is empty.    We call ${\rm Int}(H)$ the {\bf interior} of $H$. It follows from (\ref{e-int-8}) that
   \begin{eqnarray}
  {\rm Int}(H)
   &=&\{\tau\in L\mid {\rm ~for~ any~ }\sigma\in\gamma H,~\tau\cap\sigma=\emptyset\}\nonumber\\
   &=&\gamma\{\tau\in L\mid {\rm ~there~ exists~ }\sigma\in
   \gamma H {\rm  ~such ~that~ }\tau\cap\sigma\neq \emptyset\}.\label{eq-int-1}
   \end{eqnarray}
   For any $k\geq 1$, by an induction on $k$ and (\ref{eq-int-1}),
    \begin{eqnarray}
  {\rm Int}^k(H)
   &=&\gamma\{\tau_k\in L\mid {\rm~ there~ exists~ }\sigma\in
   \gamma H {\rm ~ and~ }\tau_1,\tau_2, \ldots, \tau_{k-1}\in L\nonumber \\
   &&{\rm  ~such ~that~ }\tau_1\cap\sigma\neq \emptyset {\rm ~ and~ } \tau_i\cap\tau_{i-1}\neq\emptyset{\rm ~ for~ any~ }2\leq i\leq k\}.\label{e-int-2}
   \end{eqnarray}

  \begin{lemma}\label{l10}
Let $\sigma,\sigma'$ be simplices of  $L$ with $d(\sigma,\sigma')=n$. If $s=\sigma_1,\ldots,\sigma_n$ is a path in $L$  with $\sigma_1=\sigma$, $\sigma_n=\sigma'$ and $|s|=n$, then for any $1\leq i<j\leq n$, $d(\sigma_i,\sigma_j)=j-i$.
\end{lemma}
\begin{proof}
 Suppose to the contrary, there exists $1\leq i<j\leq n$ such that $d(\sigma_i,\sigma_j)<j-i$. Let $s_{i,j}$ be the path $\sigma_i\ldots\sigma_j$ as a subset of $s$. Then $|s_{i,j}|=j-i+1$. And we can find a path $s_{i,j}'$ starting from $\sigma_i$ and ending at $\sigma_j$ with $|s'_{i,j}|<j-i+1$. Replacing $s_{i,j}$ with $s_{i,j}'$, we obtain a new path $s'$ starting from $\sigma$ and ending at $\sigma'$ with $|s'|<|s|$. This contradicts that $s$ is the path starting from $\sigma$ and ending at $\sigma'$ with the  minimal length.
\end{proof}

        \begin{lemma}\label{l-e-i}
  Let $H\in \mathcal{H}(L)$.
  \begin{enumerate}[(a). ]
  \item
  Suppose $H\neq \emptyset$. Then for any $k\geq 1$,
$ \max(L)\cap {\rm Ext}^k(H)$ is the union of all the broad paths $s=\tau_1\ldots\tau_k$ in $L$    such that  there exists $\sigma\in H$ with $\sigma\subseteq \tau_1$.
\item
Suppose $H\neq L$. Then for any $k\geq 1$,
$  \gamma{\rm Int}^k(H)$ is the union of all the paths $s'=\tau'_1\ldots\tau'_k$ in $L$    such that there exists $\sigma'\in\gamma H$ with $\tau'_1\cap \sigma' \neq\emptyset$.
  \end{enumerate}
  \end{lemma}
  \begin{proof}
Assertion  (a) follows from  (\ref{e-ext-2}).
   Assertion  (b) follows from  (\ref{e-int-2}).
  \end{proof}

In the next proposition, we list some properties of the powers of Ext and Int.

\begin{proposition}\label{th108}
Let $H\in\mathcal{H}(L)$, $k$ be a nonnegative integer and $n$ be the diameter of $L$.
  Let  $m=\max _{\sigma,\sigma'\in \max(L)}  d(\sigma,\sigma')$.
\begin{enumerate}[(a).]
\item
 If $H\neq\emptyset$ and $k\geq n$, then
${\rm Ext}^{k}(H)=L$.
\item
 If $H\neq L$ and $k\geq n$,  then ${\rm Int}^k(H)=\emptyset$.
\item
There exists  $H$  such that
$
H\subsetneq {\rm Ext}(H)\subsetneq {\rm Ext}^{2}(H)\subsetneq\ldots\subsetneq {\rm Ext}^{m-1}(H)\subsetneq L.
$
\item
There exists  $H$  such that
$
H\supsetneq {\rm Int}(H)\supsetneq {\rm Int}^{2}(H)\supsetneq\ldots\supsetneq {\rm Int}^{m-1}(H)\neq\emptyset.
$
\item
Let $k\geq 1$. Then   ${\rm Ext}^{k-1}(\gamma H)\subseteq \gamma {\rm Int}^k(H)\subseteq {\rm Ext}^{k+1}(\gamma H).$
\end{enumerate}
\end{proposition}

      \begin{proof}
 Let $k\geq n$. Then for any simplices $\sigma_1,\sigma_2\in L$, there exists a path $s$ of length $n$ starting from $\sigma_1$ and ending at $\sigma_2$.

 (a). Suppose $H\neq\emptyset$.  Then by Lemma~\ref{l-e-i}~(a), any $\sigma\in\max(L)$ is in ${\rm Ext}^k(H)$. Thus $\max(L)\cap {\rm Ext}^k(H)=L$. Thus ${\rm Ext}^k(H)=L$.

 (b). Suppose $H\neq L$.   Then by Lemma~\ref{l-e-i}~(b), any simplex $\sigma\in L$ is in $\gamma{\rm Int}^k(H)$. Thus ${\rm Int}^k(H)=\emptyset$.

 Choose $\sigma,\sigma'\in \max(L)$ such that $d(\sigma,\sigma')=m$. Choose a broad path $s=\sigma_1\ldots\sigma_m$ in $L$ with $\sigma_1=\sigma$ and $\sigma_m=\sigma'$.

(c). Let $H=\{\sigma\}$. Then by  Lemma~\ref{l10} and Lemma~\ref{l-e-i}~(a),  for any $0\leq i\leq   m-1$, $\sigma_j\in \max(L)\cap{\rm Ext}^i(H)$ for any $j\leq i+1$, and $\sigma_j\notin \max(L)\cap{\rm Ext}^i(H)$ for any $j\geq i+2$. Hence for any $0\leq i\leq m-1$, ${\rm Ext}^i(H)\subsetneq {\rm Ext}^{i+1}(H)$.

(d).
Let $H=\gamma\{\sigma\}$. Then  by  Lemma~\ref{l10} and Lemma~\ref{l-e-i}~(b),    for any $0\leq i\leq  m-1$, $\sigma_j\in {\rm Int}^i(H)$ for any $j\geq i+2$, and $\sigma_j\notin  {\rm Int}^i(H)$ for any $j\leq i+1$. Hence for any $0\leq i\leq  m-1$, ${\rm Int}^i(H)\supsetneq {\rm Int}^{i+1}(H)$.

 (e).
Let $\tau_{k-1}\in{\rm Ext}^{k-1}(\gamma H)$.  Then there exists $\sigma_{k-1}\in\max(L)\cap {\rm Ext}^{k-1}(\gamma H)$ such that $\tau_{k-1}\subseteq \sigma_{k-1}$. By Lemma~\ref{l-e-i}~(a), there exists a broad path $\sigma\ldots\sigma_{k-1}$ of length $k-1$  and $\tau\in \gamma H$ where $\sigma$ is a maximal face of $L$  such that  $\tau\subseteq\sigma$.  Consider the path $s=\sigma\ldots\sigma_{k-1}\tau_{k-1}$ of length $k$ in $L$.   Since $\sigma\cap\tau=\tau\neq\emptyset$, by Lemma~\ref{l-e-i}~(b), $\tau_{k-1}\in \gamma{\rm Int}^k(H)$.  Hence ${\rm Ext}^{k-1}(\gamma H)\subseteq \gamma{\rm Int}^k(H)$.

Let $\tau'_k\in \gamma{\rm Int}^k(H)$. By Lemma~\ref{l-e-i}~(b), there exists a path $s'=\tau'_1 \ldots\tau'_k$ in $L$ and $\sigma'\in \gamma H$ such that $\tau'_1\cap\sigma'\neq\emptyset$. Since $\Delta\sigma'\subseteq{\rm Ext}(\gamma H)$, $\tau'_1\cap\sigma' \in {\rm Ext}(\gamma H)$.  Let $\sigma'_i$ be a maximal face of $L$ such that $\tau'_i\subseteq\sigma'_i$, for each $1\leq i\leq k$. Consider the broad path $s'=\sigma'_1\ldots\sigma'_k$ of length $k$. Then since $\tau'_1\cap\sigma'\subseteq \sigma'_1$,  by Lemma~\ref{l-e-i}~(a),  $\sigma'_k\in\max(L)\cap{\rm Ext}^{k+1}(\gamma H)$. Hence $\tau'_k\in {\rm Ext}^{k+1}(\gamma H)$. Hence $\gamma{\rm Int}^k(H)\subseteq {\rm Ext}^{k+1}(\gamma H)$.
\end{proof}

  The next corollary follows from Proposition~\ref{th108}~(e).

\begin{corollary}\label{cr18}
Let $r$ be the smallest integer such that ${\rm Ext}^r(\gamma H)=L$. Let $t$ be the smallest integer  such that ${\rm Int}^t(H)=\emptyset$. Then   $t=r-1$, $r$ or $r+1$.
\end{corollary}
\begin{proof}
By Theorem~\ref{th10}~(e), for any $k\geq 1$, ${\rm Ext}^{k+1}(\gamma H)\neq L$ implies ${\rm Int}^k(H)\neq \emptyset$, and ${\rm Ext}^{k-1}(\gamma H)=L$ implies ${\rm Int}^k(H)=\emptyset$. Let $k=r-2$ and $k=r+1$ respectively, we obtain ${\rm Int}^{r-2}(H)\neq \emptyset$ and ${\rm Int}^{r+1}(H)=\emptyset$. Thus $t=r-1$, $r$ or $r+1$.
\end{proof}

The next examples show that all three cases $t=r-1$, $r$ and $r+1$ in Corollary~\ref{cr18} could happen. Hence the power-estimation in the inequality Proposition~\ref{th108}~(e) is tight.

    \begin{figure}[!htbp]
 \begin{center}
\begin{tikzpicture}
\coordinate [label=right:$v_{0,0}$]    (v_{0,0}) at (0,0);
 \coordinate [label=right:$v_{1,0}$]   (v_{1,0}) at (1,0);
 \coordinate  [label=right:$v_{2,0}$]   (v_{2,0}) at (2,0);
\coordinate  [label=right:$v_{3,0}$]   (v_{3,0}) at (3,0);
\coordinate  [label=right:$v_{4,0}$]   (v_{4,0}) at (4,0);
\coordinate  [label=right:$v_{5,0}$]   (v_{5,0}) at (5,0);
\coordinate  [label=right:$v_{6,0}$]   (v_{6,0}) at (6,0);
  \draw (v_{0,0}) -- (v_{6,0});
\coordinate [label=right: $v_{0,1}$]    (v_{0,1}) at (0.5,1*0.7);
\coordinate [label=right: $v_{1,1}$]   (v_{1,1}) at (1.5,1*0.7);
\coordinate  [label=right: $v_{2,1}$]   (v_{2,1}) at (2.5,1*0.7);
\coordinate  [label=right: $v_{3,1}$]   (v_{3,1}) at (3.5,1*0.7);
\coordinate  [label=right: $v_{4,1}$]   (v_{4,1}) at (4.5,1*0.7);
\coordinate  [label=right: $v_{5,1}$]   (v_{5,1}) at (5.5,1*0.7);
  \draw (v_{0,1}) -- (v_{5,1});
     \coordinate [label=right:$v_{0,2}$]    (v_{0,2}) at (1,2*0.7);
 \coordinate [label=right:$v_{1,2}$]   (v_{1,2}) at (2,2*0.7);
 \coordinate  [label=right:$v_{2,2}$]   (v_{2,2}) at (3,2*0.7);
\coordinate  [label=right:$v_{3,2}$]   (v_{3,2}) at (4,2*0.7);
\coordinate  [label=right:$v_{4,2}$]   (v_{4,2}) at (5,2*0.7);
  \draw (v_{0,2}) -- (v_{4,2});
       \coordinate [label=right:$v_{0,3}$]    (v_{0,3}) at (1.5,3*0.7);
 \coordinate [label=right:$v_{1,3}$]   (v_{1,3}) at (2.5,3*0.7);
 \coordinate  [label=right:$v_{2,3}$]   (v_{2,3}) at (3.5,3*0.7);
\coordinate  [label=right:$v_{3,3}$]   (v_{3,3}) at (4.5,3*0.7);
  \draw (v_{0,3}) -- (v_{3,3});
         \coordinate [label=right:$v_{0,4}$]    (v_{0,4}) at (2,4*0.7);
 \coordinate [label=right:$v_{1,4}$]   (v_{1,4}) at (3,4*0.7);
 \coordinate  [label=right:$v_{2,4}$]   (v_{2,4}) at (4,4*0.7);
  \draw (v_{0,4}) -- (v_{2,4});
           \coordinate [label=right:$v_{0,5}$]    (v_{0,5}) at (2.5,5*0.7);
 \coordinate [label=right:$v_{1,5}$]   (v_{1,5}) at (3.5,5*0.7);
  \draw (v_{0,5}) -- (v_{1,5});
             \coordinate [label=right:$v_{0,6}$]    (v_{0,6}) at (3,6*0.7);
               \draw (v_{0,0}) -- (v_{0,6});
               \draw (v_{1,0}) -- (v_{1,5});
               \draw (v_{2,0}) -- (v_{2,4});
               \draw (v_{3,0}) -- (v_{3,3});
               \draw (v_{4,0}) -- (v_{4,2});
               \draw (v_{5,0}) -- (v_{5,1});
               \draw (v_{1,0}) -- (v_{0,1});
               \draw (v_{2,0}) -- (v_{0,2});
               \draw (v_{3,0}) -- (v_{0,3});
               \draw (v_{4,0}) -- (v_{0,4});
               \draw (v_{5,0}) -- (v_{0,5});
               \draw (v_{6,0}) -- (v_{0,6});
               \fill [fill opacity=0.28][gray!100!white] (v_{0,0}) -- (v_{0,6}) -- (v_{6,0}) -- cycle;
     \end{tikzpicture}
\end{center}
 \caption{The $2$-complex $L$}
\end{figure}
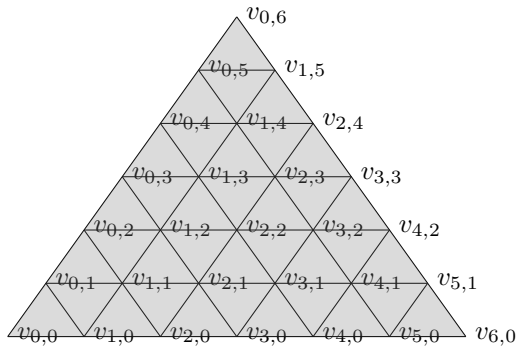

 \begin{example}
 Let $L$ be the $2$-complex with vertices $v_{i,j}$, $0\leq i,j\leq 6$, $i+j\leq 6$, given  in Figure~1.
 \begin{enumerate}
 \item
 Let $H$ be the $2$-dimensional sub-complex consisting of all the simplices inside the triangle $[v_{1,2}, v_{3,2},v_{1,4}]$, {\bf including} the boundary of $[v_{1,2}, v_{3,2},v_{1,4}]$. Then  $t=1$, $r=2$.

  \item
 Let $H$ be the $2$-dimensional sub-complex consisting of all the simplices inside the triangle $[v_{1,1}, v_{4,1},v_{1,4}]$, {\bf including} the boundary of $[v_{1,1}, v_{4,1},v_{1,4}]$. Then $t=2$, $r=2$.

  \item
  Let $H$ be the sub-hypergraph consisting of all the hyperedges inside  the triangle $[v_{1,1}, v_{4,1},v_{1,4}]$, {\bf excluding} the boundary of $[v_{1,2}, v_{4,1},v_{1,4}]$.  Then $t=2$, $r=1$.

 \end{enumerate}
 \end{example}

     \subsection{Compositions  of maps and neighborhoods}\label{s3.2}

   Let $v$ be a vertex of $L$. Recall that the {closed star} of $v$ in $L$ is the complex
  \begin{eqnarray*}
  \overline{{\rm St}}(v,L)=\Delta\{\sigma\in L\mid v\in\sigma\}.
  \end{eqnarray*}
 Let $\tau\in L$.  The {neighborhood}  of $\tau$ in $L$ is the complex
      \begin{eqnarray*}
   {\rm Nbd}(\tau)=\cup_{v\in\tau} \overline{{\rm St}}(v,L).
  \end{eqnarray*}
   By a straight-forward calculation,
   \begin{eqnarray*}
   {\rm Nbd}(\tau)&=&\cup_{v\in\tau}\Delta\{\sigma\in L\mid v\in\sigma\}\\
   &=&\Delta\cup_{v\in\tau}\{\sigma\in L\mid v\in\sigma\}\\
    &=&\Delta\{\sigma\in L\mid {\rm ~there~ exists ~}v\in \tau {\rm  ~such~ that~ } v\in\sigma\}\\
   &=&\Delta\{\sigma\in L\mid \tau\cap\sigma\neq\emptyset\}.
  \end{eqnarray*}
 Let $H\in\mathcal{H}(L)$.  The neighborhood of $H$ in $L$ is the complex
   ${\rm Nbd}(H)=\cup_{\tau\in H}{\rm Nbd}(\tau)$.
The maximal sub-hypergraph in $L$ whose neighborhood is contained in $H$ is
${\rm Nbd}^{-1}(H)=\cup_{{\rm Nbd}(H') \subseteq H} H'$.

    \begin{proposition}\label{pr-tub-1}
Let $H\in \mathcal{H}(L)$. Then
\begin{enumerate}[(a).]
\item
 ${\rm Nbd}\circ{\rm Nbd}^{-1}(H)\subseteq \delta H$;
 \item
 $\Delta H\subseteq {\rm Nbd}^{-1}\circ{\rm Nbd}(H)$;
 \item
  ${\rm Nbd}^{-1}(H)\subseteq {\rm Int}(H)$  and  the equality holds  if  for any $\sigma'\in\gamma H$, there exists $\sigma\in \gamma H$ such that $\sigma$  is maximal in $L$  and  $\sigma'\subseteq \sigma$;
 \item
 ${\rm Ext}(H)\subseteq {\rm Nbd}(H)$  and the equality holds if each vertex of $H$ is a hyperedge.
 \end{enumerate}
\end{proposition}
\begin{proof}
(a). By the definition of neighborhoods, ${\rm Nbd}\circ{\rm Nbd}^{-1}(H)\subseteq  H$. Moreover, ${\rm Nbd}\circ{\rm Nbd}^{-1}(H)$ is a simplicial complex. Hence ${\rm Nbd}\circ{\rm Nbd}^{-1}(H)\subseteq \delta H$.

(b). By the definition of neighborhoods,   any maximal face of   ${\rm Nbd}(H)$ is a maximal face of $L$. Thus ${\rm Nbd}(H)$ is completely determined by $\max(L)\cap{\rm Nbd}(H)$.  In order to prove (b),
we only need to show ${\rm Nbd}(\Delta H)={\rm Nbd}(H)$. Since $H\subseteq \Delta H$,   ${\rm Nbd}(H)\subseteq {\rm Nbd}(\Delta H)$.
Let $\sigma'\in \max (L)\cap{\rm Nbd}(\Delta H)$. Then there exists $\sigma\in \Delta H$ such that $\sigma\cap\sigma'\neq\emptyset$.  Moreover, there exists $\tau\in H$ such that $\sigma\subseteq\tau$. Hence $\tau\cap\sigma'\neq\emptyset$. Hence $\sigma'\in\max(L)\cap{\rm Nbd}(H)$.  Thus $\max(L)\cap {\rm Nbd}(\Delta H)\subseteq \max(L)\cap{\rm Nbd}(H)$.  Thus  $ {\rm Nbd}(\Delta H)\subseteq   {\rm Nbd}(H)$.   Therefore, ${\rm Nbd}(\Delta H)={\rm Nbd}(H)$. Consequently, $\Delta H\subseteq {\rm Nbd}^{-1}\circ{\rm Nbd}(H)$.

(c). By a straight-forward calculation,
\begin{eqnarray}
{\rm Nbd}^{-1}(H)
&=&\{\tau\in L\mid {\rm Nbd}(\tau)\subseteq H\}\nonumber\\
&=&\{\tau\in L\mid {\rm~  for~ any ~}\sigma'\in\gamma H,~ \sigma'\notin\Delta\{\sigma\mid \tau\cap\sigma\neq\emptyset\}\}\nonumber\\
&=&\{\tau\in L\mid {\rm ~ for~ any~ }\sigma'\in\gamma H{\rm ~ and~ any ~ }\sigma\in L \nonumber\\
&&{\rm ~ with~ }\tau\cap\sigma\neq\emptyset, ~ \sigma'\notin\Delta\sigma\}\nonumber\\
&=&\{\tau\in L\mid {\rm~  for~ any~ }\sigma'\in\gamma H{\rm ~ and ~any ~ }\sigma'\subseteq \sigma,   ~ \tau\cap\sigma=\emptyset\}\nonumber\\
&\subseteq &\{\tau\in L\mid {\rm ~ for~ any~ }\sigma'\in\gamma H,~\tau\cap\sigma'=\emptyset\}\nonumber\\
&=&{\rm Int} (H).  \label{e-tub9}
\end{eqnarray}
Suppose  in addition that   for any $\sigma'\in\gamma H$, there exists $\sigma\in \gamma H$ such that $\sigma$  is maximal in $L$ and   $\sigma'\subseteq\sigma$.  Then the  equality holds in the  penultimate inequality of (\ref{e-tub9}).   Thus  ${\rm Nbd}^{-1}(H)= {\rm Int}(H)$.

(d). By a straight-forward calculation,
\begin{eqnarray}
{\rm Nbd}(H)&=& \cup_{\tau\in H}\Delta\{\sigma\in L\mid \tau\cap\sigma\neq\emptyset\}\nonumber\\
&\supseteq& \cup_{\tau\in H}\Delta\{\sigma\in L \mid \tau\subseteq\sigma\}\label{e-tub-8}\\
&=&\cup_{\tau\in H}\Delta\{\sigma\in \max(L)\mid \tau\subseteq\sigma\}\nonumber\\
&=&{\rm Ext}(H). \nonumber
\end{eqnarray}
 Suppose in addition that each vertex of $H$ is a hyperedge. Then
   the equality holds in (\ref{e-tub-8}).  Hence  ${\rm Nbd}(H)={\rm Ext}(H)$.
   \end{proof}

   The next corollary follows from Proposition~\ref{pr-tub-1}.

   \begin{corollary} \label{co-tub-1}
   Let $H\in \mathcal{H}(L)$. Then
   \begin{enumerate}[(a).]
   \item
   If   for any $\sigma'\in\gamma H$, there exists $\sigma\in \gamma H$ such that $\sigma$  is maximal in $L$  and  $\sigma'\subseteq\sigma$,  then  ${\rm Ext}\circ {\rm Int}(H)\subseteq \delta H$;

   \item
   If each vertex of $H$ is a hyperedge, then $\Delta H\subseteq {\rm Int}\circ {\rm Ext}(H)$.

   \end{enumerate}
   \end{corollary}

  \section{Map algebras on random hypergraphs and random simplicial complexes}\label{sec4}

  In this section, we study the  maps $D\Delta$, $D\delta$ and  $D\gamma$   as well as their compositions acting on $D(\mathcal{H}(L))$ and $D(\mathcal{K}(L))$. We prove Theorem~\ref{th10} and Theorem~\ref{th888}.

Let $w\in G$. For any $f\in D(\mathcal{H}(L))$ and any $H\in \mathcal{H}(L)$, it follows from (\ref{e-intro1}) that
\begin{eqnarray}\label{e-4.1-3}
Dw(f)(H)=\sum_{w(H')=H}f(H').
\end{eqnarray}
  Let $t$ be a positive integer and let $W\in G^t$. Suppose $W=(\ldots (w_1*w_2)*\ldots* w_t)\in G^t$  with any $(t-2)$-brackets $(\cdot )$  giving the order of evaluations, $*=\wedge$ or $+$, and $w_1,w_2, \ldots,w_t\in G$.  Then with the help of (\ref{e-intro2}), $W$ induces a map
\begin{eqnarray}\label{e-4.1-2}
DW: D\big(\mathcal{H}(L)\big)^{\times t}\longrightarrow D\big(\mathcal{H}(L)\big)
\end{eqnarray}
given by
\begin{eqnarray}\label{e-4.1-1}
DW(f_1,f_2,\ldots,f_t)(H)
=\sum_{(\ldots(H_1*H_2)*\ldots *H_k)=H}\prod_{i=1}^t\big(Dw_i(f_i)(H_i)\big)
\end{eqnarray}
for any $(f_1,\ldots,f_t)\in D\big(\mathcal{H}(L)\big) ^{\times t}$ and any $H\in\mathcal{H}(L)$.
The next lemma shows that (\ref{e-4.1-1}) gives a well-defined map (\ref{e-4.1-2}).

\begin{lemma}\label{p12}
For and $t\geq 1$ and any $(f_1,f_2, \ldots,f_t)\in D\big(\mathcal{H}(L)\big) ^{\times t}$, $DW(f_1,f_2,\ldots,f_t)\in D\big(\mathcal{H}(L)\big)$.
\end{lemma}
\begin{proof}
To prove Lemma~\ref{p12}, we need to prove
\begin{eqnarray}\label{e-pr1}
\sum_{H\in\mathcal{H}(L)}DW(f_1,f_2,\ldots,f_t)(H)=1
\end{eqnarray}
for any $t\geq 1$.
Firstly, we prove (\ref{e-pr1}) for $t=1$. Since $w$ is a self-map on $\mathcal{H}(L)$, for any $f_1\in D(\mathcal{H}(L))$,
\begin{eqnarray*}
\sum_{H_1\in \mathcal{H}(L)} Dw_1(f_1)(H_1) = \sum_{H_1\in\mathcal{H}(L)}\sum_{w_1(H'_1)=H_1}f_1(H'_1)
 = \sum_{H'_1\in\mathcal{H}(L)}f_1(H'_1).
\end{eqnarray*}
 Thus (\ref{e-pr1}) holds for $t=1$.
Secondly, we use induction on $t$ and prove (\ref{e-pr1}) for $t\geq 2$. By (\ref{e-4.1-3}) and (\ref{e-4.1-1}), we have
\begin{eqnarray*}
&&\sum_{H\in\mathcal{H}(L)}DW(f_1,f_2,\ldots,f_t)(H)\\
&=&\sum_{H_1,H_2,\ldots, H_t\in\mathcal{H}(L)}\prod_{i=1}^t\left(\sum_{w_i(H_i')=H_i}f_i(H_i')\right)\\
&=&\sum_{H_1,H_2,\ldots, H_{t-1}\in\mathcal{H}(L)}\prod_{i=1}^{t-1}\left(\sum_{w_i(H_i')=H_i}f_i(H_i')\right)\left(\sum_{H_t\in\mathcal{H}(L)}\sum_{w_t(H'_t)=H_t}f_t(H'_t)\right)\\
&=&\sum_{H_1,H_2,\ldots, H_{t-1}\in\mathcal{H}(L)}\prod_{i=1}^{t-1}\left(\sum_{w_i(H_i')=H_i}f_i(H_i')\right).
\end{eqnarray*}
By an induction on $t$, (\ref{e-pr1}) follows.
   \end{proof}

   Now we prove Theorem~\ref{th10}.

   \begin{proof}[Proof of Theorem~\ref{th10}]
    Theorem~\ref{th10}~(a), (b), (c), (d) follow from Proposition~\ref{th108}~(a), (b), (c), (d) respectively.  Theorem~\ref{th10}~(e) follows from Proposition~\ref{th108}~(e). Theorem~\ref{th10}~(f) follows from Corollary~\ref{co-tub-1}~(b).
      \end{proof}

Let  $p$ be a function from $L$ to $[0,1]$.

  \begin{lemma}\label{p1}
The map $D\Delta$ sends $\bar{{\rm P}}_{L,p}\in D(\mathcal{H}(L))$ in Definition~\ref{def3}  to  a  random  simplicial complex  $D\Delta (\bar{{\rm P}}_{L,p})  \in  D(\mathcal{K}(L))$   given by
\begin{eqnarray*}
D\Delta(K)= \Big(\prod_{\tau\in {\rm max}(K)} p(\tau)\Big)\Big(\prod_{\tau\notin  K}\big(1-p(\tau)\big)\Big)
\end{eqnarray*}
for any  $K\in \mathcal{K}(L)$.
 \end{lemma}

  \begin{proof}
  Let  $K$ be  a  sub-simplicial complex  of  $L$.  Let ${\rm max}(K)$  be  the collection of all the maximal  faces  in  $K$.    Let $S=\{\sigma_1,\sigma_2,\ldots,\sigma_s\}$  be  any set  of  distinct  simplices in  $K$ such that     $s$  is  a  non-negative  integer  and for  each $\sigma_i$,  $i=1,2,\ldots,s$, there exists $\tau\in {\rm max}(K)$ such that $\sigma_i\subsetneq \tau$.  Here $S$  is allowed to be the emptyset.  Suppose  $S$  runs  over all such sets  of  simplices  in $K$.  Then
  \begin{eqnarray*}
  H={\rm max}(K)\sqcup S
  \end{eqnarray*}
  runs over all the  sub-hypergraphs of $L$  such that  $\Delta H=K$.  Consequently,
  \begin{eqnarray*}
 && {\rm Prob}[\Delta  H=K{\rm~for~the~random~hypergraph~}H\sim \overline{P}_{L,p}(H)]\\
  &=&\sum_{\Delta  H=K}\overline{P}_{L,p}(H)\\
  &=&\sum_{\Delta  H=K~~~}\prod_{\sigma\in H}p(\sigma)\prod_{\sigma\notin H}\big(1-p(\sigma)\big)\\
  &=&\sum_{H={\rm max}(K) \sqcup S~~~}\prod_{\sigma\in H}p(\sigma)\prod_{\sigma\notin H}\big(1-p(\sigma)\big)\\
  &=&\sum_{S\subseteq  K\setminus \max(K)~~}\prod_{\sigma\in\max(K)} p(\sigma) \prod_{\sigma\in S}p(\sigma)\prod_{\sigma\in K\setminus (\max(K)\sqcup S)} \big(1-p(\sigma)\big)\prod_{\sigma\notin K}\big(1-p(\sigma)\big)\\
  &=&\Big(\prod_{\sigma\in\max(K)} p(\sigma)\Big)\Big(\sum_{S\subseteq  K\setminus \max(K)~~}\prod_{\sigma\in S}p(\sigma)\prod_{\sigma\in K\setminus (\max(K)\sqcup S)} \big(1-p(\sigma)\big)\Big)\Big(\prod_{\sigma\notin K}\big(1-p(\sigma)\big)\Big)\\
  &=&\Big(\prod_{\tau\in {\rm max}(K)} p(\tau)\Big)\Big(\prod_{\tau\notin  K}\big(1-p(\tau)\big)\Big).
  \end{eqnarray*}
  Here the last equality follows from that
  \begin{eqnarray*}
  &&\sum_{S\subseteq  K\setminus \max(K)~~}\prod_{\sigma\in S}p(\sigma)\prod_{\sigma\in K\setminus (\max(K)\sqcup S)} \big(1-p(\sigma)\big)\\
  &=&\prod_{\sigma\in K\setminus \max(K)}\Big(p(\sigma)+\big(1-p(\sigma)\big)\Big)=\prod_{\sigma\in K\setminus \max(K)}1
  =1.
  \end{eqnarray*}
We obtain  the  lemma.
\end{proof}

\begin{lemma}\label{p9}
The map $D\delta$ sends $\bar{{\rm P}}_{L,p}$ in Definition~\ref{def3}  to    a  random  simplicial complex  $D\delta (\bar{{\rm P}}_{L,p})  \in  D(\mathcal{K}(L))$   given by
\begin{eqnarray*}
D\delta(K)= \sum_{\delta H=K} \prod_{\sigma\in H} p(\sigma) \prod_{\sigma\notin H} \big(1-p(\sigma)\big)
\end{eqnarray*}
for any  $K\in \mathcal{K}(L)$.
 \end{lemma}

\begin{proof}
Let  $K$   be  a   sub-simplicial  complex  of  $L$.   Then
\begin{eqnarray*}
&&{\rm Prob}[\delta  H=K{\rm~for~the~random~hypergraph~}H\sim \overline{P}_{L,p}(H)]\\
&=&\sum_{\delta H=K} \overline {P}_{L,p}\\
&=&\sum_{\delta H=K} \prod_{\sigma\in H} p(\sigma) \prod_{\sigma\notin H} \big(1-p(\sigma)\big).
\end{eqnarray*}
By the definition of $D\delta$,  the lemma follows.
\end{proof}

  \begin{lemma}\label{p8}
The map $D\gamma$ sends $\bar{{\rm P}}_{L,p}$ to $\bar{{\rm P}}_{L,1-p}$.
 \end{lemma}

 \begin{proof}
Let $H$ be a hypergraph in $L$. Let $\tau\in L$. Then the probability that $\tau$ is a hyperedge in $H$  is $p(\tau)$. Thus the probability that $\sigma$ is a hyperedge in $\gamma H$ is $1-p(\tau)$. By the construction of the random hypergraphs in Definition~\ref{def3}, the probability function of $\gamma H$ is $\bar{{\rm P}}_{L,r,1-p}$.
 \end{proof}

 Let  $p',p''$ be  functions from $L$ to $[0,1]$.

   \begin{lemma}\label{p28}
   The map $D\cap$ sends the pair $(\bar{{\rm P}}_{L,p'},\bar{{\rm P}}_{L,p''})$ to $\bar{{\rm P}}_{L,p'p''}$. And the map $D\cup$ sends the pair $(\bar{{\rm P}}_{L,p'},\bar{{\rm P}}_{L,p''})$ to $\bar{{\rm P}}_{L,1-(1-p')(1-p'')}$.
  \end{lemma}

  \begin{proof}
We choose  hypergraphs $H',H''\in \mathcal{H} (L)$  independently at random    with probability functions $\bar{{\rm P}}_{L,p'}$   and $\bar{{\rm P}}_{L,p''}$  respectively.  In order to prove Lemma~\ref{p28}, we need to show
\begin{enumerate}[(a).]
\item
the random hypergraph $ H'\cap H''$   satisfies Definition~\ref{def3}   with probability function $\bar{{\rm P}}_{L,p'p''}$;
\item
the random hypergraph $ H'\cup H''$   satisfies Definition~\ref{def3}   with probability function $\bar{{\rm P}}_{L,1-(1-p')(1-p'')}$.
\end{enumerate}
Let $\sigma\in L$.  Consider two independent trials: (1). generate $H'$; (2). generate $H''$.

{\sc Proof of} (a). $\sigma\in H'\cap H''$ if and only if $\sigma\in H'$ in trial (1) and $\sigma\in H''$ in trial (2).  Thus $\sigma\in H'\cap H''$ has probability $pp'$. Letting $\sigma$ run over $L$, these trials of  $\sigma$'s are independent.

{\sc Proof of} (b).  $\sigma\notin H'\cup H''$ if and only if $\sigma\notin H'$ in trial (1) and $\sigma\notin H''$ in trial (2).  Thus $\sigma\notin H'\cup H''$ has probability $(1-p')(1-p'')$, and $\sigma\in H'\cap H''$ has probability $1-(1-p')(1-p'')$. Letting $\sigma$ run over $L$, these trials of  $\sigma$'s are independent.
\end{proof}

    Now we prove Theorem~\ref{th888}.

    \begin{proof}[Proof of Theorem~\ref{th888}]
    Theorem~\ref{th888}~(a) follows from Lemma~\ref{p8}.     Theorem~\ref{th888}~(b) follows from Lemma~\ref{p1}.     Theorem~\ref{th888}~(c) follows from Lemma~\ref{p9}.        Theorem~\ref{th888}~(d) follows from the first assertion of Lemma~\ref{p28}. Theorem~\ref{th888}~(e) follows from the second assertion of Lemma~\ref{p28}.
    \end{proof}

    \section*{Acknowledgments}
{\small
The  authors would like to express their  deepest
gratitude to the editor and the referee for their careful reading and helpful suggestion.  The project was supported in part by the Singapore Ministry of Education research grant
(AcRF Tier 1 WBS No. R-146-000-222-112).  Shiquan Ren was supported in part by
the National Research Foundation, Prime Minister's Office, Singapore (CREATE programme), Natural Science Foundation of China (NSFC grant no. 12001310), and China Postdoctoral Science Foundation.  Chengyuan  Wu was supported in part by the President's Graduate Fellowship of National University
of Singapore.  Jie Wu was supported by Natural Science Foundation of China (NSFC grant no. 11971144), High-level Scientific Research Foundation of Hebei Province and the start-up research fund from BIMSA. }

\bigskip

Shiquan Ren

 Address:  School of Mathematics and Statistics, Henan University,  Kaifeng 475004,  China.
  e-mail:  renshiquan@henu.edu.cn

\bigskip

Chengyuan Wu

Address: Department of Mathematics, National University of Singapore,  119076, Singapore.
  e-mail: wuchengyuan@u.nus.edu
  
\bigskip

Jie  Wu

Address: Yanqi Lake Beijing Institute of Mathematical Sciences and Applications, Beijing 101408, China.
  e-mail: wujie@bimsa.cn


\begin{thebibliography}{99}

\bibitem{yd-4}
L. Aronshtam and N. Linial, \emph{The threshold for $d$-collapsibility in random complexes}, {Random Struct. Algor.} {\bf 48} (2016), 260-269.


\bibitem{8}
L. Aronshtam, N. Linial, T. Luczak and R. Meshulam, \emph{Collapsibility and vanishing of top homology in random simplicial complexes. }
{ Discrete Comput. Geom. }{\bf{49(2)}} (2013), 317-334.


\bibitem{6}
E. Babson, C. Hoffman and M. Kahle, \emph {The fundamental group of random $2$-complexes. } {J. Amer. Math. Soc.} {\bf 24(1)} (2010), 1-28.


\bibitem{Berge}
C. Berge,  \emph{Graphs and hypergraphs. } American Elsevier Pub. Co. North-Holland, New York, 1976.

\bibitem{y2-4}
D. Cohen, A. Costa, M. Farber and T. Kappeler, \emph{Topology of random $2$-complexes}. {Discrete Comput. Geom. } {\bf 47} (2012), 117-149.

\bibitem{y2-5}
D. Cohen, A. Costa, M. Farber and T. Kappeler, \emph{Correction to Topology of random $2$-complexes}. {Discrete Comput. Geom. } {\bf 56} (2016), 502-503.


\bibitem{y2-1}
A.E. Costa and M. Farber, \emph{The asphericity of   random $2$-dimensional complexes}, {Random Struct. Algor.} {\bf 46} (2015), 261-273.

\bibitem{y2-2}
A.E. Costa and M. Farber, \emph{Geometry and topology of random $2$-complexes}, {Israel J. Math.} {\bf 209} (2015), 883-927.

\bibitem{cfh}
A. Costa, M. Farber and D. Horak, \emph{Fundamental groups of clique complexes of random graphs}, Trans. London Math. Soc. {\bf 2(1)} (2015), 1-32.

\bibitem{m1}
A. Costa and M. Farber, \emph{Large random simplicial complexes, I}, J. Topol. Anal. {\bf 8(3)} (2016), 399-429.

\bibitem{m2}
A. Costa and M. Farber, \emph{Large random simplicial complexes, II; the fundamental group}, J. Topol. Anal. {\bf 9(3)} (2017), 441-483.


\bibitem{m3}
A. Costa and M. Farber, \emph{Large random simplicial complexes, III; the critical dimension}, J. Knot Theory Ramifications {\bf 26(2)} (2017), 1740010-1 - 1740010-26.

\bibitem{m4}
A. Costa and M. Farber, \emph{Random simplicial complexes}, Configuration Spaces  129-153, Springer INdAM Series {\bf  14} (2016), Springer, 129-153.




\bibitem{1959er}
P. Erd\"os and A. R\'enyi, \emph{On random graphs I}, {Publ. Math. Debrecen} {\bf 6} (1959), 290-297.

\bibitem{1960er}
P. Erd\"os and A. R\'enyi, \emph{On the evolution of random graphs}, {Publ. Math. Inst. Hungar. Acad. Sci.} {\bf 5} (1960), 17-61.










\bibitem{1959g}
E.N. Gilbert, \emph{Random graphs}, {Ann. Math. Statist. } {\bf 30(4)} (1959), 1141-1144.




\bibitem{yd-2}
A. Gundert, \emph{On eigenvalues of random complexes}, {Israel J. Math.} {\bf 216} (2016), 545-582.

\bibitem{yd-3}
A. Gundert and U. Wagner, \emph{On topological minors in random simplicial complexes}, {Proc. Amer. Math. Soc.} {\bf 144} (2016), 1815-1828.


\bibitem{hatcher}
{A. Hatcher,} \emph{Algebraic Topology.} Cambridge University Press, Cambridge, 2002.



\bibitem{9}
C. Hoffman, M. Kahle and E. Paquette, \emph{The threshold for integer homology in random $d$-complexes}, {Discrete Comput. Geom.} {\bf 57} (2017), 810-823.


\bibitem{clique}
M. Kahle, \emph{Topology of random clique complexes}, Discrete Math. {\bf 309(6)} (2009), 1658-1671.

\bibitem{contem}
M. Kahle, \emph{Topology of random simplicial complexes: a survey}, Algebraic Topology: Applications and New Directions, Contem. Math. {\bf 620} (2014), 201-221.

\bibitem{yd-1}
M. Kahle and B. Pittle, \emph{Inside the critical window for cohomology of random $k$-complexes}, {Random Struct. Algor.} {\bf 48} (2016), 102-104.




\bibitem{7}
D. N. Kozlov, \emph {The threshold function for vanishing of the top homology group of random $d$-complexes. } {Proc. Amer. Math. Soc. } {\bf 138(12)} (2010), 4517-4527.


\bibitem{y2-3}
N. Linial and R. Meshulam, \emph{Homological connectivity of random $2$-complexes}, {Combinatorica} {\bf 26} (2006), 475-487.

\bibitem{annals}
N. Linial and Y. Peled, \emph{On the phase transition in random simplicial complexes}, {Ann. Math.} {\bf 184} (2016), 745-773.






\bibitem{yd-5}
R. Meshulam and N. Wallach, \emph{Homological connectivity of random $k$-dimensional complexes}, {Random Struct. Algor.} {\bf 34} (2009), 408-417.



\bibitem{epide}
R. Pastor-Satorras,
C. Castellano, P.V. Mieghem, and  A. Vespignani, \emph{Epidemic processes in complex networks}, Rev. Mod. Phys. {\bf 87 } (2015), 925-979.

\end{thebibliography}
\end{document}